\documentclass[a4paper,11pt]{article}
\usepackage{amssymb}
\usepackage[utf8]{inputenc}
\usepackage[T1]{fontenc}
\usepackage{graphicx}
 \usepackage{amssymb,amsmath,latexsym,amsthm}
\usepackage[a4paper]{geometry}
\usepackage{url}
\usepackage{tocloft}  
\usepackage[utf8]{inputenc}   
\usepackage[T1]{fontenc}
\usepackage{graphicx}
\usepackage{mathrsfs}
\usepackage{frcursive}
\usepackage{color}
\usepackage{calligra}
\usepackage{pst-node}
\usepackage{breqn}
\usepackage{tikz}
\usepackage[mathscr]{euscript}
\usepackage{calrsfs}
\usepackage{float}
\usepackage{amssymb}
\usepackage{amssymb}
\usepackage{titlesec}
\usepackage[Rejne]{fncychap}
\usepackage{lipsum}
\usepackage{float}
\usepackage{xfrac} 
\usepackage[Rejne]{fncychap}
\usepackage{faktor}
\usepackage{amsmath}
\usepackage[T1]{fontenc}
\usepackage[utf8]{inputenc}
\usepackage[colorlinks=true,urlcolor=black,linkcolor=blue,citecolor=black]{hyperref}
\usepackage[english]{babel}
\usepackage{amsthm}
\usepackage{bbm}
\usepackage{float}

\newcounter{mythm}[section]

\newtheorem{tm}[mythm]{Theorem}

\newtheorem{defi}[mythm]{Definition}

\newtheorem{rem}[mythm]{Remark}
\newtheorem{cor}[mythm]{Corollary}
\newtheorem{ex}[mythm]{Example}
\newtheorem{nota}[mythm]{Notation}

\def\C{\mathcal{C}}
\def\R{\mathbb{R}}

\def\N{\mathbb{N}}

\newcommand\restr[2]{{
  \left.\kern-\nulldelimiterspace 
  #1 
  \littletaller 
  \right|_{#2} 
  }}

\newcommand{\littletaller}{\mathchoice{\vphantom{\big|}}{}{}{}}
 
\begin{document}

\begin{center}
\textsc{\textbf{\LARGE{A Liouville theorem for some asymptotically conical Calabi-Yau manifolds}}}
\end{center}

\begin{center}
\textsc{ABDOU OUSSAMA BENABIDA}
\end{center}

\noindent \textsc{Abstract.} Let $(\C, J_\C, \omega_\C, g_\C)$ be a Calabi-Yau cone and $(M, J, \omega, g)$ an open Ricci-flat Kähler manifold. We show that, if there exists a diffeomorphism $\Phi: \C \setminus \overline{B_1(o)} \rightarrow M \setminus K$, for some compact $K \subset M$, such that $\Phi^{*}J$ is asymptotic to $J_\C$ and $C^{-1} \omega_{\C} \leq \Phi^{*} \omega \leq C \omega_{\C}$ for some $C \geq 1$, then $(M, g)$ is asymptotically conical (AC) with tangent cone at infinity given by $(\C, d_{g_\C})$. As a consequence, we obtain that any Ricci-flat Kähler metric on $T^{*}S^n$ which is quasi-isometric to the Stenzel metric \cite{stenzel_ricci-flat_1993} must be equal to the Stenzel metric up to scaling and diffeomorphism. Similarly, any Ricci-flat Kähler metric on $\mathcal{O}_{\mathbb{P}^1}(-1)^{\oplus2}$ which is quasi-isometric to the Candelas-De la Ossa metric \cite{candelas_comments_1990} must be equal to the Candelas-De la Ossa metric up to scaling and diffeomorphism. This provides new examples of complete Calabi-Yau manifolds for which a Liouville-type theroem holds.

\vspace{1cm}

\section{Introduction}
Liouville-type theorems are a fundamental theme in the analysis of elliptic partial differential equations. Broadly speaking, they provide rigidity results for solutions satisfying appropriate growth conditions, and they are frequently used as a key tool in proving regularity results. In the context of complex Monge-Ampère equations, the classical Liouville theorem for Kähler metrics on 
$\mathbb{C}^n$ dates back to Riebesehl-Schulz \cite{MR735051} where it is proved that any Kähler metric $\omega$ on $\mathbb{C}^n$ which has a constant determinant i.e. $\omega^n = \omega_{\text{euc}}^n$ and which is quasi-isometric to the Euclidean metric i.e. $A^{-1} \omega_{\text{euc}} \leq \omega \leq A \omega_{\text{euc}}$ for some $A \geq 1$ is the pullback of the Euclidean metric by an element of $GL(n, \mathbb{C})$. Hein \cite{https://doi.org/10.1002/cpa.21751} generalized this result to product manifolds of the form $\mathbb{C}^n \times Y$ where $Y$ is a closed Calabi-Yau manifold. A different proof using a mean value formula is given by Li-Li-Zhang \cite{LLZ}. \\

Recently, Klemmensen \cite{klemmensen2025liouville} obtained a Liouville theorem for Calabi-Yau cones which states that any cscK metric $\omega$ which is quasi-isometric to a Calabi-Yau cone metric $\omega_\C$ must be the pullback of $\omega_\C$ by a an automorphism of the cone commuting with the scaling.\\

In this paper, we prove that if $(\C, g_\C, \omega_\C)$ is a Calabi-Yau cone and  $(M,g, \omega)$ is a complete Calabi-Yau manifold with complex structure asymptotic, at infinity, to that of $\C$ such that, outside a compact, we have 
$$ A^{-1} \omega_{\C} \leq \omega \leq A \omega_{\C}, $$
then $(M,g)$ is asymptotically conical with a unique tangent cone at infinity given by $(\C, d_{g_\C})$.\\

To do so, we use the regularity of the complex Monge-Ampère equation to extract a $C_{\text{loc}}^{k, \beta}$ convergent subsequence of $\varepsilon_i^2 g$, where $\varepsilon_i \rightarrow 0$ and then we use Klemmensen's Liouville theorem to conclude that the limit is the pullback of $g_\C$ by a an automorphism of the cone commuting with the scaling. This defines the same metric space as $(\C, g_\C)$. We, then, use the $C_{\text{loc}}^{k, \beta}$ convergence to prove that $g$ has quadratic curvature decay. Using Sun-Zhang's result \cite{AC}, this implies that $g$ must be asymptotically conical and hence fits into the Conlon-Hein classification \cite{conlon_classification_2024}. See Theorem \ref{AC} for the exact statement and proof.\\

As a corollary, we are able to use the uniqueness result from Conlon-Hein \cite[Theorem C]{conlon_classification_2024} to obtain a Liouville theorem for some examples of asymptotically conical Calabi-Yau manifolds: the smoothing $T^{*}S^n$ and the small resolution $\mathcal{O}_{\mathbb{P}^1}(-1)^{\oplus2}$ (case $n=3$) of the nodal cone $\C := \left\{ z\in \mathbb{C}^{n+1}  : \sum_{i=1}^n z_i^2 = 0\right\}$. See Corollary \ref{coro}. 

\subsection*{Acknowledgement}
The author is grateful to his PhD supervisor, Frédéric Rochon, for bringing several relevant references to his attention. He also thanks Hans-Joachim Hein for introducing him to the topic and for valuable discussions related to the problem. The author is further grateful to Shouhei Honda for explaining a proof of the quadratic curvature decay based on results of Colding-Minicozzi and Cheeger-Colding.

\section{Preliminaries}
In this section, we introduce the necessary definitions and review earlier results that will be important for our work.
\subsection{Calabi-Yau cones}
\begin{defi} \label{conical}
The \textbf{Riemannian cone} over a given closed connected Riemannian manifold $\left(L, g_L\right)$ is the Riemannian manifold $\left(\C, g_\C\right)$ where $\C=\mathbb{R}_{>0} \times L$ and $g_\C=d r^2+r^2 g_L$, with $r: \C \rightarrow \mathbb{R}_{>0}$ the projection onto the first factor. We often write $\C=\C(L)$ and call $L$ the link of $\C$. 
\end{defi}
\begin{rem}
    Sometimes we also include the vertex of the cone i.e. a point denoted by $o$ attached to the cone at $\{r=0\}$.
\end{rem}

\begin{defi} \label{CY cone}
    A \textbf{Calabi-Yau cone} is a simply connected Riemannian cone $(\C, g_\C)$ together with a $g_\C$-parallel complex structure $J_\C$ such that $g_\C$ is $J_\C$-Kähler and Ricci-flat. We denote the associated Kähler form by $\omega_\C$ i.e. $\omega_\C (X,Y) = g_\C(J_\C X, Y)$.
\end{defi}
\begin{rem}
    Throughout this paper, we will use Kähler metric and Kähler form interchangeably.
\end{rem}
\begin{nota}
    We denote by $\text{Aut}_{\text{Scl}} (\C)$ the space of holomorphic automorphisms of $\C$ which commute with the scalings $r \mapsto \lambda r$ for $\lambda >0$. Equivalently, these are the holomorphic automorphisms of $\C$ which commute with the scaling vector field $r \partial_r$ and the Reeb vector field $\xi := J_\C (r \partial_r)$.
\end{nota}

The following are some examples of Calabi-Yau cones
\begin{ex}
\hfill
    \begin{itemize}
        \item A trivial example is given by $\mathbb{C}^n$ with the Euclidean metric as a cone over $S^{2n-1}$.
        \item More generally, if $G$ is a finite subgroup of $SU(n)$ acting freely on $\mathbb{C}^n \setminus \{ 0 \}$, then the quotient $\mathbb{C}^n / G$ with the induced quotient metric is a Calabi-Yau cone. 
        \item The nodal cone $\C := \left\{ z\in \mathbb{C}^{n+1}  : \sum_{i=1}^n z_i^2 = 0\right\}$ is a Calabi-Yau cone. In \cite{stenzel_ricci-flat_1993}, Stenzel constructed a Ricci-flat Kähler cone metric given by
        $$\omega_{\C} = i \partial \overline{\partial} \left( \| z\|^2 \right)^{\frac{n-1}{n}}.$$
        We refer to a singularity which is isomorphic to the quadric $C$ above as a \emph{nodal singulairty}.
        \item The Calabi ansatz \cite{calabi_metriques_1979} is a general construction for regular Calabi-Yau cones. If $D$ is a Kähler-Einstein Fano manifold, then, for every integer $k>0$ dividing $c_1(D)$, there exist a Ricci-flat Kähler cone metric on $(\frac{1}{k} K_D)^{\times}$, the blowdown of the zero section of $\frac{1}{k} K_D$.
    \end{itemize}
\end{ex}
In \cite{klemmensen2025liouville}, Klemmensen proved the following Liouville theorem for Calabi-Yau cones
\begin{tm}
    Let $(\C, \omega_\C)$ be a Calabi-Yau cone and let $\omega$ be another constant scalar curvature Kähler metric on $\C$ satisfying 
    $$
    A^{-1} \omega_\C \leq \omega \leq A \omega_\C,
    $$
    for some $A \geq 1$. Then, $\omega = \Psi^{*} \omega_\C$ for some $\Psi \in \text{Aut}_{\text{Scl}} (\C)$. 
\end{tm}
\begin{rem}
    When $\C = \mathbb{C}^n$, we get that $\text{Aut}_{\text{Scl}} (\mathbb{C}^n) = GL(n, \mathbb{C})$ and therefore we recover the classical Liouville theorem of  Riebesehl-Schulz \cite{MR735051}.
\end{rem}
\subsection{Asymptotically conical Calabi-Yau manifolds}
We give the following definitions
\begin{defi}
    Let $(M,g)$ be an open Riemannian manifold. We say that $M$ is \textbf{asymptotically conical} (AC) if there exists a Riemannian cone $(\C, g_\C)$, a compact $K \subset M$ and a diffeomorphism $\Phi : \C \setminus \overline{B_1(o)} \rightarrow M \setminus K$ such that 
    $$\left|\nabla_{g_{\C}}^{j} \left(\Phi^{*} g - g_{\C} \right)\right|_{g_{\C}} = O(r^{-\lambda -j}) $$ for some $\lambda >0$.
    
\end{defi}
\begin{defi}
    An \textbf{asymptotically conical Calabi-Yau manifold} is an open Ricci-flat Kähler manifold $(M,g, \omega)$ such that $g$ is asymptotically conical.
\end{defi}
There have been several works constructing examples of asymptotically Calabi-Yau manifolds \cite{goto_calabi-yau_2012,van_coevering_ricci-flat_2010, van2011examples, conlon2013asymptotically, conlon2015asymptotically}, but Conlon-Hein \cite{conlon_classification_2024} obtained the following complete classification.
\begin{tm}[Conlon-Hein] \label{CH}
    We fix a Calabi-Yau cone $(\C, \omega_\C)$. Suppose $V$ is an affine variety which is a deformation of $\C$ with negative $\xi$-weight and let $\pi : M \rightarrow V$ be a holomorphic crepant resolution such that $M$ is Kähler. Then, for any class $\mathrm{t} \in H^2(M, \R)$ such that $\langle\mathrm{t}^d , Z \rangle >0$ for all irreducible subvarieties $Z$ of $\mathrm{Exc} (\pi)$, $d = \dim Z >0$, and for all $g \in \mathrm{Aut}_{\text{Scl}}(C)$, $M$ admits an asymptotically conical Ricci-flat Kähler metric which is asymptotic to $g^{*}\omega_\C$. In addition, these classify all asymptotically conical Calabi-Yau manifolds up to diffeomorphism.
\end{tm}

See \cite[Definition 1.7]{conlon_classification_2024} for the definition of a deformation with negative $\xi$-weight. In this paper, we are interested in the following two examples
\begin{ex}
\hfill
\begin{itemize}
    \item 
The nodal cone from above $\C = \left\{ z\in \mathbb{C}^{n+1}  : \sum_{i=1}^n z_i^2 = 0\right\}$ has a natural smoothing given by 
$$ \C_{1}:= \left\{ z\in \mathbb{C}^{n+1}  : \sum_{i=1}^n z_i^2 = 1\right\}.$$
This is in fact isomorphic to the cotangent bundle of the sphere $T^{*}S^n$. In \cite{stenzel_ricci-flat_1993}, the author obtained an asymptotically conical Calabi-Yau metric on $T^{*}S^n$ of the form 
$$\omega_{\text{st}}= i \partial \overline{\partial}(f(||z||^2)),$$
where $f$ is a function solving a certain ODE. 
\item In the case when $n=3$, the nodal cone also has a small crepant resolution given by the total space of the bundle
$$ \mathcal{O}_{\mathbb{P}^1}(-1)^{\oplus2} \rightarrow \mathbb{P}^1.$$
In \cite{candelas_comments_1990}, the authors obtained an asymptotically conical Calabi-Yau metric on $\mathcal{O}_{\mathbb{P}^1}(-1)^{\oplus2}$ of the form 
$$ \omega_{\text{co}} = i \partial \overline{\partial} (f(r^3)) + 4 \pi^{*}\omega_{FS},$$
where $\omega_{FS}$ is the Fubini-Study metric on $\mathbb{P}^1$, $\pi: \mathcal{O}_{\mathbb{P}^1}(-1)^{\oplus2} \rightarrow \mathbb{P}^1 $ is the bundle projection and $f$ solves a certain ODE.
 \end{itemize}   
\end{ex}
\section{Main result and proof}
The following is our main theorem
\begin{tm} \label{AC}
Let $(M, J)$ be an open complex Kähler manifold with trivial canonical bundle modeled on a Calabi-Yau cone $(\C,J_{\C}, \omega_{\C}, g_{\C})$ i.e. there exists a compact $K \subset M$ and a diffeomorphism $\Phi : \C \setminus \overline{B_1(o)} \rightarrow M \setminus K$ such that $$\left|\nabla_{g_{\C}}^{j} \left(\Phi^{*} J - J_{\C} \right)\right|_{g_{\C}} = O(r^{-\lambda -j}) $$ for some $\lambda >0$. Let $\omega$ be a Ricci-flat $J$-Kähler form on $M$, with associated metric $g$, such that $$ C^{-1} \omega_{\C} \leq \Phi^{*} \omega \leq C \omega_{\C}, $$
for some $C \geq 1$. Then $g$ is asymptotically conical with tangent cone at infinity given by $(\C, d_{g_\C}, o)$.
\end{tm}

\begin{proof}
    Since $g_{\C}$ is a cone metric, it has maximal volume growth and therefore, by our bound hypothesis, $g$ has maximal volume growth as well. By the result of Sun-Zhang \cite{AC}, if $g$ has quadratic curvature decay, then it must be asymptotically conical. Hence, the theorem follows from the following two claims\\

    \textbf{Claim 1} $(M,g)$ has a unique tangent cone at infinity given by $\left(\C, d_{g_\C}, o\right)$. Moreover, the convergence is, in fact, in $C_{\text{loc}}^{k, \beta}$ for any $0<\beta<1$ and $k \in \N$.\\

    \textbf{Claim 2} The metric $g$ has quadratic curvature decay i.e. there exists $C >0$ such that for a fixed $p \in M$ and all $x \in M$ we have
    $$ |Rm_g(x)| \leq C \cdot r(x)^{-2},$$
    where $r(x):= d(p,x)$.\\
    
    \emph{Proof of Claim 1.}
    Throughout the proof, C designates a uniform constant that may vary at different occurrences without further mention. Fix $p \in M \setminus K$ and $\{\varepsilon_i\}$ a sequence of positive numbers such that $\varepsilon_i \rightarrow 0 $ as $i \rightarrow \infty$. Consider the sequence of metrics $g_i := \varepsilon_i^{2} g$. On $\C \setminus \overline{ B_{\varepsilon_i}(o)}$, we define $\tilde g_i := \delta_{\frac{1}{\varepsilon_i}}^{*} (\Phi^{*} g_i)$, where $\delta_\lambda : \C \rightarrow \C$ is the cone dilation $r \mapsto \lambda r$.  Therefore, $\Phi \circ \delta_{\frac{1}{\varepsilon_i}}$ induces an isometry $\left(M \setminus K, g_i, p \right) \cong \left(\C \setminus \overline{B_{\varepsilon_i}(o)}, \tilde g_i, p_i \right)$, where $p_i := (\Phi \circ \delta_{\frac{1}{\varepsilon_i}})^{-1} (\{p\})$. We wish to extract a subsequence of $\left(\C \setminus B_{\varepsilon_i}(o), \tilde g_i \right)$ which converges in the $C_{\text{loc}}^{k, \beta}$ topology to $\left(\C \setminus \{ o\}, \Psi^{*}g_{\C} \right)$ for some $\Psi \in \text{Aut}_{\text{Scl}} (\C)$ and then use that to conclude that $(M, d_{g_i}, p)$ subconverges in the pointed Gromov-Hausdorff sense to $(\C, d_{g_\C}, o)$.\\
        
        First, for $\varepsilon>0$, we denote the pullback of the complex structure $J_\varepsilon := \delta_{\frac{1}{\varepsilon}}^{*} (\Phi^{*} J) $ and $J_i:= J_{\varepsilon_i}$. Therefore $\tilde g_i$ is $J_i$-Kähler and we denote the associated Kähler form by $\tilde \omega_i= \varepsilon_i^2 \delta_{\frac{1}{\varepsilon_i}}^{*} \left( \Phi^{*} \omega \right)$.
\begin{itemize}
    \item \emph{Step 1: $J_\varepsilon$ converges to $J_{\C}$ in $C_{\text{loc}}^{\infty} (\C \setminus \{o \})$ as $\varepsilon \to 0$:}\\
    
Let
\[
T := \Phi^{*}J - J_{\C}
\qquad \text{on } \C \setminus \overline{B_1(o)}.
\]
By assumption, for every $j \ge 0$ there exists $C_j>0$ such that
\[
|\nabla_{g_{\C}}^{\,j} T|_{g_{\C}} \le C_j r^{-\lambda-j}
\qquad \text{on } \C \setminus \overline{B_1(o)}.
\]

\medskip

First observe that the cone complex structure is invariant under dilations:
\[
\delta_a^{*} J_{\C} = J_{\C} \qquad \text{for all } a>0.
\]
Therefore
\[
J_\varepsilon - J_{\C}
= \delta_{1/\varepsilon}^{*}(\Phi^{*}J)
  - \delta_{1/\varepsilon}^{*} J_{\C}
= \delta_{1/\varepsilon}^{*}(\Phi^{*}J - J_{\C})
= \delta_{1/\varepsilon}^{*} T.
\]

\medskip

Next recall that the cone metric satisfies
\[
\delta_a^{*} g_{\C} = a^2 g_{\C} .
\]
Consequently, for any tensor field $S$ of type $(1,1)$ and any $j\ge0$,
\[
\bigl|\nabla_{g_{\C}}^{\,j} (\delta_a^{*}S)\bigr|_{g_{\C}}(x)
\le
a^{j}\,
\bigl|\nabla_{g_{\C}}^{\,j} S\bigr|_{g_{\C}}(\delta_a x).
\]

Applying this to $S=T$ and $a=1/\varepsilon$ gives
\[
\begin{aligned}
|\nabla_{g_{\C}}^{\,j}(J_\varepsilon - J_{\C})|(x)
&=
|\nabla_{g_{\C}}^{\,j}(\delta_{1/\varepsilon}^{*}T)|(x) \\
&\le
\left(\frac{1}{\varepsilon}\right)^{j}
|\nabla_{g_{\C}}^{\,j} T|(\delta_{1/\varepsilon}x).
\end{aligned}
\]

Using the decay assumption on $T$ yields
\[
|\nabla_{g_{\C}}^{\,j}(J_\varepsilon - J_{\C})|(x)
\le
\left(\frac{1}{\varepsilon}\right)^{j}
C_j r(\delta_{1/\varepsilon}x)^{-\lambda-j}.
\]

The cone dilation scales the radial function by
\[
r(\delta_{1/\varepsilon}x) = \frac{1}{\varepsilon} r(x),
\]
hence, we obtain
\[
r(\delta_{1/\varepsilon}x)^{-\lambda-j}
=
\left(\frac{1}{\varepsilon}\right)^{-\lambda-j}
r(x)^{-\lambda-j}
=
\varepsilon^{\lambda+j} r(x)^{-\lambda-j}.
\]

Substituting this gives
\[
|\nabla_{g_{\C}}^{\,j}(J_\varepsilon - J_{\C})|(x)
\le
C_j \varepsilon^{\lambda} r(x)^{-\lambda-j}.
\]

\medskip

Let $K' \subset \C \setminus \{o\}$ be compact. Then there exists $r_0>0$ such that
$r(x) \ge r_0$ on $K'$. For $\varepsilon$ sufficiently small we have
$K' \subset \C \setminus \overline{B_{\varepsilon}(o)}$, and therefore
\[
\sup_{x\in K'}
|\nabla_{g_{\C}}^{\,j}(J_\varepsilon - J_{\C})|(x)
\le
C_j \varepsilon^{\lambda} r_0^{-\lambda-j}
\longrightarrow 0
\qquad \text{as } \varepsilon \to 0.
\]

Thus
\[
J_\varepsilon \to J_{\C}
\quad \text{in } C^\infty_{\mathrm{loc}}(\C \setminus \{o\})
\]
with respect to the cone metric $g_{\C}$.
    \item \emph{Step 2: $\tilde \omega_i$ subconverges to some $\tilde \omega_{\infty}$ in $C^{k, \beta'}_{\text{loc}}(\C \setminus \{o \})$, for $k \in \N$ and $0<\beta'<1$:}\\
    
By homogeneity of the cone metric,
\[
\delta_{1/\varepsilon_i}^*\omega_{\C}=\varepsilon_i^{-2}\omega_{\C}.
\]
Hence from
\[
C^{-1}\,\omega_{\C} \le \Phi^*\omega \le C\,\omega_{\C}
\]
we obtain
\[
C^{-1}\,\omega_{\C} \le \tilde\omega_i \le C\,\omega_{\C}
\qquad\text{on } \C\setminus \overline{B_{\varepsilon_i}(o)}.
\]
Moreover $\tilde\omega_i$ is $J_i$-K\"ahler and Ricci-flat.\\

Let $K'' \Subset \C\setminus\{o\}$ be a compact. Since the cone metric has bounded geometry on
compact sets away from the apex, $K''$ can be covered by finitely many $J_\C$-holomorphic coordinate charts 
\[
\Psi_{\alpha}:B_1(0)\subset\mathbb{C}^n\to \C
\]
such that
$$  \frac{1}{C} \omega_{\text{euc}} \leq \Psi_\alpha^{*} \omega_\C \leq C \omega_{\text{euc}}$$
for some uniform constant $C$. Therefore, we also get
$$ \frac{1}{C} \omega_{\text{euc}} \leq \Psi_\alpha^{*} \tilde \omega_i\leq  C \omega_{\text{euc}}. $$
Now, the family of complex structures $\{ J_\varepsilon\}_{\varepsilon>0}$ is continuous, with all its derivatives, in $\varepsilon$ on the compact $K''$ and, by Step $1$, $J_\varepsilon \to J_{\C}$ smoothly as $\varepsilon \to 0$. Therefore, by the Nirenberg-Newlander theorem with parameter (see \cite[Proposition 1.2.]{NNP}), $K''$ can be covered by
finitely many $J_i$-holomorphic coordinate charts
\[
\Psi_{\alpha,i}:B_1(0)\subset\mathbb{C}^n\to \C
\]
such that $\Psi_{\alpha,i} \rightarrow \Psi_\alpha$ in $C_{\text{loc}}^{\infty}$. Therefore, in these coordinates, the pulled-back metrics
\[
\omega_{\alpha,i}:=\Psi_{\alpha,i}^*\tilde\omega_i
\]
are Ricci-flat K\"ahler metrics on $B_1(0)\subset \mathbb{C}^n$ satisfying
\[
\frac{1}{C}\omega_{\mathrm{euc}}\le\omega_{\alpha,i}\le C\,\omega_{\mathrm{euc}}
\]
for some uniform constant $C$ and $i \gg 1$. Therefore, the Evans-Krylov estimates imply that
\[
\|\omega_{\alpha,i}\|_{C^{0,\beta}(B_1)}\le C
\]
for $0<\beta < 1$. By differentiating the complex Monge-Ampère equation, using the Schauder estimates and a Bootstrapping argument, we get
\[
\|\omega_{\alpha,i}\|_{C^{k,\beta}(B_1)}\le C_k.
\] By \cite[Lemma 6.33]{gilbarg1998elliptic},
after passing to a subsequence, for $0<\beta'<\beta$,
\[
\tilde\omega_i \to \tilde\omega_\infty
\qquad\text{in } C^{k,\beta'}(K'').
\]
A diagonal argument over an exhaustion of $\C\setminus\{o\}$ by compact sets
gives the desired $C^{k,\beta'}_{\mathrm{loc}}$ convergence.
    
    \item \emph{Step 3: $\tilde \omega_{\infty}$ is a smooth Ricci-flat $J_{\C}$-Kähler form:}\\

    Since $J_i\to J_{\C}$ smoothly on compact subsets and each $\tilde\omega_i$ is
$J_i$-K\"ahler, passing to the limit in the identities
\[
d\tilde\omega_i=0,
\qquad
\tilde\omega_i(J_i\cdot,J_i\cdot)=\tilde\omega_i(\cdot,\cdot),
\]
shows that
\[
d\tilde\omega_\infty=0,
\qquad
\tilde\omega_\infty(J_{\C}\cdot,J_{\C}\cdot)=\tilde\omega_\infty(\cdot,\cdot).
\]
Moreover, from the uniform bound
\[
C^{-1}\omega_{\C}\le \tilde\omega_i\le C\omega_{\C}
\]
passing to the limit, we obtain
\[
C^{-1}\omega_{\C}\le \tilde\omega_\infty\le C\omega_{\C},
\]
so $\tilde\omega_\infty$ is positive. Thus $\tilde\omega_\infty$ is a
$C^{k,\beta'}$ $J_{\C}$-K\"ahler form.

For smoothness, again, we can use the regularity of the complex Monge-Ampère equation and use a Bootstrapping argument. Moreover, Since $\tilde\omega_i$ is Ricci-flat and $J_i \to J_{\C}$ smoothly,
passing to the limit in the identity
\[
\mathrm{Ric}_{J_i}(\tilde\omega_i)
=
-\sqrt{-1}\partial_{J_i}\bar\partial_{J_i}
\log(\tilde\omega_i^n)=0
\]
shows that
\[
\mathrm{Ric}_{J_{\C}}(\tilde\omega_\infty)=0.
\]
    \item \emph{Step 4: $\tilde \omega_{\infty} = \Psi^{*} \omega_{\C}$ for some $\Psi \in \text{Aut}_{\text{Scl}} ({\C})$:}\\

    $\tilde \omega_\infty$ is a smooth Ricci-flat $J_{\C}$-K\"ahler form on $\C \setminus \{ o\}$ such that 
$$
C^{-1} \omega_{\C} \leq \tilde \omega_{\infty} \leq C \omega_{\C}.
$$
Therefore, using the Liouville theorem for Calabi-Yau cones \cite[Theorem 2.3]{klemmensen2025liouville}, we get that $\tilde \omega_\infty = \Psi^{*} \omega_{\C}$ for some $\Psi \in \text{Aut}_{\text{Scl}} ({\C})$.
    
    \item \emph{Step 5: $\left(M, d_{g_i}, p \right)$ subconverges in the pointed Gromov-Hausdorff sense to $\left(\C, d_{g_{\C}}, o \right)$:}\\

    First, since $K \subset M$ is compact and $p \in M \setminus K$, there exists $D >0$ such that $K \subset B_g(p,D)$, hence $K \subset B_{g_i}(p, \varepsilon_i D)$. Consider the quotient pointed sequence of metric spaces $(\hat{M}, \hat d_{g_i}, p)$, where $\hat M:= M /K \cong (M \setminus K) \bigcup \{ *_K\}$. Therefore, since $\varepsilon_i D \rightarrow 0$, the pointed sequence $(M, d_{g_i}, p)$ has the same sublimits as $(\hat M, \hat d_{g_i},p)$. But, we also know that $(M \setminus K, g_i,p) \cong (\C \setminus \overline{B_{\varepsilon_i}(o)}, \tilde g_i, p_i)$ and $
r(p_i)=\varepsilon_i r(x)\to 0
$, where $x:= \Phi(p)$, therefore, since, up to taking a subsequence, we have $\tilde g_i\to \Psi^*g_{\C}$ in $C^{k, \beta}$ on compact subsets of $C\setminus\{o\}$ and $\bigl(\C,d_{\Psi^*g_{\C}},o\bigr)$ is isometric, as a metric space, to $\bigl(\C,d_{g_{\C}},o\bigr)$, we conclude that
\[
\bigl(\hat M, \hat d_{g_i},p\bigr)
\to
\bigl(\C,d_{g_{\C}},o\bigr)
\]
in the pointed Gromov-Hausdorff sense. Hence, $(M, d_{g_i}, p) \rightarrow (\C, d_{g_\C}, o)$ in the pointed Gromov-Hausdorff sense. 
\hfill \qedsymbol

        \end{itemize}
         \emph{Proof of Claim 2.} In the following, we suppose, without loss of generality that $p \in M \setminus K$. Suppose by contradiction that $g$ does not have quadratic curvature deacy. Therefore, for all $j \in \N$, there exists $x_j \in M$ such that 
$$ |Rm_g(x_j)| > j \cdot r(x_j)^{-2}.$$
First, notice that, up to taking a subsequence, we have 
$$r(x_j) \rightarrow \infty.$$ Else, there exists $C'>0$ such that $r(x_j) \leq C'$, or, equivalently $x_j \in \overline{B(p, C')}$. But, $|Rm_g|$ is bounded on the compact $\overline{B(p, C')}$, and therefore $r(x_j)^2 |Rm_g(x_j)| \leq C''$ which contradicts our supposition.\\ 

Now, define $\lambda_j := r(x_j)$ and let $g_j:= \lambda_j^{-2} g$. In particular, we have $r_{g_j}(x_j):= d_{g_j}(x_j,p) = 1$. Since, $\lambda_j \rightarrow \infty$, then, by Claim $1$, $(M,\tilde g_j)$ subconverges in $C_{\text{loc}}^{k, \beta}$ to the cone $(\C, g_{\C})$, where $\tilde g_j=\delta_{\lambda_j}^{*} (\Phi^{*} g_j)$. Denote $q_j := \delta_{\lambda_j}^{-1} (\Phi^{-1} (x_j))$ and $r_{g_\C}(x):= d_{g_\C}(o,x)$. Using the quasi-isometry, we have
\begin{align*}
C^{-1} r(x_j) -C' \leq & r_{g_\C}(\Phi^{-1}(x_j)) \leq C r(x_j) +C' \implies
\\
    \frac{C^{-1}}{2} \leq \lambda_j^{-1} \left( C^{-1} r(x_j) - C' \right) \leq & \lambda_j^{-1} r_{g_\C}(\Phi^{-1}(x_j)) = r_{g_\C}(q_j) \leq \lambda_j^{-1} \left(C r(x_j) + C'\right) \leq 2 C.
    \end{align*}
In other words, $r_{g_\C}(q_j)= d_{g_C}(o, q_j)$ remains bounded and therefore, up to passing to a subsequence, we have $q_j \rightarrow q_\infty$.  Hence, by the $C_{\text{loc}}^{k, \beta}$ convergence, we get 
$$|Rm_{g_j}(x_j)|= |Rm_{\tilde g_j}(q_j)| \rightarrow |Rm_{g_{\C}} (q_\infty)|, \quad j \rightarrow \infty.$$
Therefore $|Rm_{g_j}(x_j)|$ is uniformly bounded which contradicts the fact that
$$ |Rm_{g_j}  (x_j)| = \lambda_j^2 |Rm_g(x_j)| > j.$$
         
    \end{proof}
    \begin{rem}
        One could also obtain the quadratic curvature decay by using the uniqueness of the tangent cone guaranteed by the result of Colding-Minicozzi \cite{uniqueness} and the $C_{\text{loc}}^{\infty}$ convergence obtained by Cheeger-Colding \cite[Theorem 7.3]{Cheeger_Colding} and proceeding in a similar way as in the proof of Claim $2$. This would prove a note from \cite[p3]{uniqueness}.
    \end{rem}
Now, since $(M,J,g)$ is an asymptotically conical Calabi-Yau manifold, then, by \cite[Theorem A, Theorem B]{MR4740213} (See Theorem \ref{CH}), it is equivalent, up to diffeomorphism, to a K\"ahler crepant resolution of a negative $\xi$-weight deformation of the cone $\C$ (See \cite[Definition 1.7]{MR4740213} for the definition of negative $\xi$-weight deformation of the cone). Moreover, the uniqueness result \cite[Theorem C]{MR4740213} allows us to obtain a Liouville theorem for smoothings of the nodal cone as well as its crepant resolution when $n=3$. 
\begin{cor} \label{coro}
\hfill
\begin{itemize}
    \item 
    Consider the AC Calabi-Yau manifold
    $$
T^{*}S^n \cong \{ z \in \mathbb{C}^{n+1};  \hspace{0.1cm} \sum_{i=1}^{n+1} z_i^2 =1 \}
$$ 
with the Stenzel metric $\omega_{\text{st}}$ from \cite{stenzel_ricci-flat_1993}. If $\omega$ is another Ricci-flat K\"ahler metric on $T^{*}S^n$ such that 
$$
A^{-1}\omega_{\text{st}} \leq \omega \leq A \omega_{\text{st}}, 
$$
for some $A>1$, then $\omega= \lambda  \Psi^*\omega_{\text{st}}$ for some diffeomorphism $\Psi$ and $\lambda>0$.
\item Consider the AC Calabi-Yau manifold
$$
\mathcal{O}_{\mathbb{P}^1}(-1)^{\oplus2}
$$
with the Candelas-De la Ossa metric $\omega_{co}$ from \cite{candelas_comments_1990}.  If $\omega$ is another Ricci-flat K\"ahler metric on $\mathcal{O}_{\mathbb{P}^1}(-1)^{\oplus2}$ such that 
$$
A^{-1}\omega_{\text{co}} \leq \omega \leq A \omega_{\text{co}}, 
$$
for some $A>1$, then $\omega= \lambda \Psi^*\omega_{\text{co}}$ for some diffeomorphism $\Psi$ and $\lambda>0$.
\end{itemize}
\end{cor}
\begin{proof}
    This is a combination of Theorem \ref{AC} and \cite[Theorem C]{MR4740213}.
\end{proof}

\medskip
\textsc{Département de mathématiques, Université du Québec à Montréal}\\
\textit{Email address :} \texttt{benabida.oussama@gmail.com}
\end{document}